\documentclass[12pt]{amsart}
\pdfoutput=1 
\usepackage{amssymb}
\usepackage{amsmath}
\usepackage{amscd}
\usepackage[all]{xy}
\usepackage{tikz}

\usetikzlibrary{intersections, calc, patterns}

\theoremstyle{plain}
\newtheorem{theorem}{Theorem}[section]
\newtheorem{corollary}[theorem]{Corollary}

\newtheorem{lemma}[theorem]{Lemma}
\newtheorem{proposition}[theorem]{Proposition}

\theoremstyle{definition}

\newtheorem{example}[theorem]{Example}

\theoremstyle{remark}
\newtheorem{remark}[theorem]{Remark}

\def\Z{{\mathbb Z}}
\def\N{{\mathbb N}}
\def\R{{\mathbb R}}
\def\C{{\mathbb C}}
\def\Q{{\mathbb Q}}

\def\cG{{\mathcal G}}

\def\cS{{\mathcal S}}

\def\a{{\boldsymbol a}}
\def\b{{\boldsymbol b}}

\def\e{{\boldsymbol e}}
\def\g{{\boldsymbol g}}

\def\s{{\boldsymbol s}}

\def\u{{\boldsymbol u}}
\def\v{{\boldsymbol v}}
\def\w{{\boldsymbol w}}
\def\x{{\boldsymbol x}}

\def\z{{\boldsymbol z}}
\def\0{{\boldsymbol 0}}
\def\1{{\boldsymbol 1}}

\def\bbeta{{\boldsymbol{\beta}}}

\def\supp{{\rm supp}}
\def\nsupp{{\rm nsupp}}

\def\ord{{\rm ord}}

\def\Ker{{\rm Ker}}

\def\ini{{\rm in}}

\def\rank{{\rm rank}}

\def\NS{{\rm NS}}

\title{Logarithmic $A$-hypergeometric series}
\author{Mutsumi Saito}

\begin{document}

\begin{abstract}
The method of Frobenius is a standard technique to construct series solutions
of an ordinary linear differential equation around a regular singular point.
In the classical case, 
when the roots of the indicial polynomial are separated by 
an integer, logarithmic solutions can be constructed
by means of perturbation of a root.

The method for a regular $A$-hypergeometric system is a theme
of the book by Saito, Sturmfels, and Takayama.
Whereas they perturbed a parameter vector to obtain 
logarithmic $A$-hypergeometric series solutions,
we adopt a different perturbation
in this paper.


\end{abstract}

\maketitle

\section{Introduction}
The method of Frobenius is a standard technique to construct series solutions
of an ordinary linear differential equation around a regular singular point.
In the classical case, 
when the roots of the indicial polynomial are separated by 
an integer, logarithmic solutions can be constructed
by means of perturbation of a root (cf. e.g. \cite{Ince}).

Let $A=(\a_1,\ldots, \a_n)=(a_{ij})$ be a $d\times n$-matrix
of rank $d$ with coefficients in $\Z$. 
Throughout this paper,
we assume the homogeneity of $A$, i.e.,
we assume that all $\a_j$ belong to one hyperplane off the origin 
in ${\Q}^d$.
Let ${\N}$ be the set of nonnegative integers.
Let $I_A$ denote the toric ideal in the polynomial ring
${\C}[\partial]={\C}[\partial_1,\ldots,\partial_n]$, i.e.,
\begin{equation}
\label{eqn:ToricIdeal}
I_A=\langle \partial^\u-\partial^\v\, :\, A\u=A\v, \, \u, \v\in {\N}^n
\rangle \subseteq {\C}[\partial].
\end{equation}
Here and hereafter we use the multi-index notation;
for example, $\partial^\u$ means $\partial^{u_1}\cdots \partial^{u_n}$
for $\u=(u_1,\ldots,u_n)^T$.
Given a column vector $\bbeta=(\beta_1,\ldots,\beta_d)^T\in {\C}^d$,
let $H_A(\bbeta)$ denote the left ideal of the Weyl algebra
$$
D={\C}\langle x_1,\ldots, x_n,\partial_1,\ldots,
\partial_n\rangle
$$
generated by $I_A$ and 
\begin{equation}
\label{eqn:EulerOperators}
\sum_{j=1}^n a_{ij}\theta_{j} -\beta_i\qquad
(i=1,\ldots, d),
\end{equation}
where $\theta_j =x_j\partial_j$.
The quotient $M_A(\bbeta)=D/H_A(\bbeta)$ is called
the {\it $A$-hypergeometric system with parameter $\bbeta$},
and a formal series annihilated by $H_A(\bbeta)$
an {\it $A$-hypergeometric series with parameter $\bbeta$}.
The homogeneity of $A$ is known to be equivalent to
the regularity of $M_A(\bbeta)$ by Hotta \cite{Hotta} and
Schulze, Walther \cite{SW}.

For a generic parameter $\bbeta$, Gel'fand, Graev, Kapranov, and Zelevinsky 
\cite{GGZ}, \cite{GZK2} constructed
series solutions to $M_A(\bbeta)$.
More generally, Saito, Sturmfels, and Takayama \cite{SST}
constructed logarithm-free series solutions, which we will review in Section 2.
Then they perturbed a parameter vector to construct logarithmic series solutions
\cite[\S 3.5]{SST}.
This is reasonable, 
because this method perturbed the easier equations \eqref{eqn:EulerOperators}
keeping the difficult ones \eqref{eqn:ToricIdeal} unchanged.
Then we can easily obtain perturbed solutions.
However, we need to make a suitable linear combination of
perturbed solutions before taking a limit, and
it is not clear how to describe this linear combination
(except the unimodular case \cite[\S 3.6]{SST}).

Set
\begin{equation}
L:=\Ker_\Z(A)=
\{ \u\in \Z^n\,|\, A\u=\0\}.
\end{equation}
We know that the logarithmic coefficients of $A$-hypergeometric
series solutions are polynomials of $\log x^\b$ $(\b\in L)$
\cite[Proposition 5.2]{LogFree}.
In this paper, we adopt a perturbation by elements of $L$.
At first glance, it does not seem a good idea, 
because we perturb a solution of
the difficult
equations \eqref{eqn:ToricIdeal} keeping the easier ones 
\eqref{eqn:EulerOperators}.
But it turns out that we can evaluate the order of perturbation, and
we can explicitly describe logarithmic series solutions 
(Theorems \ref{thm:1}, \ref{thm2} and Remarks \ref{rmk:Method1}, 
\ref{rmk:Method2}).

In \cite{AS}, Adolphson and  Sperber treated
$A$-hypergeometric series solutions with logarithm
mainly of degree $1$ or $2$, and considered an application
to mirror symmetry.

This paper is organized as follows.
In Sections 2 and 3, following \cite{SST},
we recall some notions on $A$-hypergeometric series.
In particular, we recall fake exponents and negative supports in Section 2, 
and then
we recall that the
fake exponents can be computed 
by the standard pairs of the initial ideal of $I_A$
in Section 3.

In Section 4, for a generic weight $\w$ and a fake exponent $\v$,
we define a set $\NS_\w(\v)$ of negative supports, over which
we consider a series.
Then we introduce a Gale dual and a hyperplane arrangement
 to visualize negative supports.

In Sections 5 and 6, we state the main results
(Theorems \ref{thm:1}, \ref{thm2} and Remarks \ref{rmk:Method1}, 
\ref{rmk:Method2}).
We consider a perturbation by a single element 
and several elements of $L$ in Sections 5 and 6, respectively.
We make one section for the single element case, because it is much easier
to consider.
In both sections, we first consider orders of perturbations
(Lemma \ref{CoefOrder}, Corollaries \ref{cor:CoefOrder},
\ref{cor:a_u:CoefOrder}, and Lemma \ref{lem:a_u:multi}).
Then we see that 
the perturbed solution operated by the difficult ones
\eqref{eqn:ToricIdeal}
has some positive orders,
and we can prove the main results.

Throughout this paper we run three examples (Examples \ref{running1-1},
\ref{running2-1}, and \ref{running3-1}) to illustrate the theory.

\section{logarithm-free canonical $A$-hypergeometric series}

In this section, we recall logarithm-free canonical $A$-hypergeometric
series. For details, see \cite{SST}.

Fix a generic weight vector $\w=(w_1,\ldots,w_n)\in {\R}^n$.
The ideal of the polynomial ring ${\C}[\theta]=
{\C}[\theta_1,\ldots,\theta_n]$ defined by
\begin{equation}
\label{def:fin}
\widetilde{{\rm fin}}_\w(H_A(\bbeta))
:=D\cdot {\rm in}_\w(I_A)\cap {\C}[\theta]+\langle A\theta-\bbeta\rangle
\end{equation}
is called the {\it fake indicial ideal},
where ${\rm in}_\w(I_A)$ denotes the initial ideal of $I_A$ with respect
to $\w$, and
$\langle A\theta-\bbeta\rangle$ denotes the ideal generated by
$\sum_{j=1}^n a_{ij}\theta_j -\beta_i$ ($i=1,\ldots, d$).
Each zero of $\widetilde{{\rm fin}}_\w(H_A(\bbeta))$
is called a {\it fake exponent}.

Let $\w\cdot \u$ denote
$w_1u_1+\cdots +w_nu_n$ for $\u\in {\Q}^n$.
An $A$-hypergeometric series
\begin{equation}
\label{eq:SeriesInW}
x^\v\cdot\sum_{\u\in L}
g_\u(\log x)x^\u
\qquad (g_\u\in {\C}[x]:=\C[x_1,\ldots,x_n])
\end{equation}
is said to be {\it in the direction of $\w$} if
there exists
a basis $\u^{(1)},\ldots, \u^{(n)}$ of
${\Q}^n$ with $\w\cdot \u^{(j)}>0$ 
($j=1,\ldots,n$) such that
$g_\u=0$ whenever $\u\notin \sum_{j=1}^n {\Q}_{\geq 0}\u^{(j)}$.
A fake exponent $\v$ is called an {\it exponent}
if there exists
an $A$-hypergeometric series (\ref{eq:SeriesInW}) 
in the direction of $\w$ 
with nonzero $g_\0$.
Let $\prec$ be the lexicographic order on ${\N}^n$.
Suppose that $\u^{(1)},\ldots, \u^{(n)}$ is a basis as above.
Then a monomial like $x^\v\cdot {\rm in}_\prec (g_\0)(\log x)$ 
in the $A$-hypergeometric series 
\begin{equation}
\label{eq:CanonicalSeries}
x^\v\cdot\sum_{\u\in L\cap \sum_{j=1}^n {\Q}_{\geq 0}\u^{(j)}}
g_\u(\log x)x^\u
\qquad (g_\u\in {\C}[x])
\end{equation}
with nonzero $g_\0$
is called a
{\it starting monomial}.
The $A$-hypergeometric 
series (\ref{eq:CanonicalSeries})
is said to be {\it canonical} with respect to
$\w$ if no starting monomials other than
$x^\v\cdot {\rm in}_\prec (g_0)(\log x)$ appear in the series.

Next we recall logarithm-free $A$-hypergeometric series $\phi_\v$.
For $\v\in {\C}^n$, its {\it negative support} 
${\rm nsupp}(\v)$ is the set of indices
$i$ with $v_i\in {\Z}_{<0}$.
When ${\rm nsupp}(\v)$ is minimal with respect to inclusions
among ${\rm nsupp}(\v+\u)$ with $\u\in L$,
$\v$ is said to have {\it minimal negative support}.
For $\v$ satisfying $A\v=\bbeta$ with minimal negative support,
we define a formal series
\begin{equation}
\label{eqn:CanonicalSeries}
\phi_\v=x^\v\cdot
\sum_{\u\in N_\v}\frac{[\v]_{\u_-}}{[\v+\u]_{\u_+}}x^{\u},
\end{equation}
where 
\begin{equation}
N_\v=\{\, \u\in L\, |\,
 {\rm nsupp}(\v)={\rm nsupp}(\v+\u)\,\},
\end{equation}
and
$\u_+, \u_-\in {\N}^n$ satisfy $\u=\u_+ -\u_-$ with disjoint supports,
and
$[\v]_\u=\prod_{j=1}^n[v_j]_{u_j}=\prod_{j=1}^n v_j(v_j-1)\cdots (v_j-u_j+1)$
for $\u\in {\N}^n$.
Proposition 3.4.13 and Theorem 3.4.14 in \cite{SST}
respectively state that
the series $\phi_\v$ is $A$-hypergeometric, and that
if $\v$ is a fake exponent of $M_A(\bbeta)$, then
$\phi_\v$ is canonical, and $\v$ is an exponent.

\section{Standard pairs and fake exponents}

In this section, we review standard pairs and fake exponents
following \cite{SST}.

Let $M$ be a monomial ideal in $\C[\partial]
=\C[\partial_1,\ldots \partial_n]$.
A pair $(\a, \sigma)$ $(\a\in \N^n, \sigma\subseteq [1,n])$
is {\it standard} if it satisfies
\begin{enumerate}
\item[\rm (1)]
$a_i=0$ for all $i\in \sigma$.
\item[\rm (2)]
For any $\b\in \N^\sigma$, $\partial^\a\partial^\b\notin M$.
\item[\rm (3)]
For any $l\notin \sigma$, there exists $\b\in \N^{\sigma\cup\{ l\}}$
such that $\partial^\a\partial^\b\in M$.
\end{enumerate}

Let $\cS(M)$ denote the set of standard pairs of $M$.
Then, by \cite[Corollary 3.2.3]{SST}, 
$\v$ is a fake exponent of $M_A(\bbeta)$ with respect to
$\w$ if and only if $A\v=\bbeta$ and
there exists a standard pair $(\a, \sigma)\in \cS(\ini_\w(I_A))$
such that
$v_j=a_j$ for all $j\notin \sigma$.

\begin{example}[cf. Example 3.5.3 in \cite{SST}]
\label{running1-1}
Let 
$A=
\begin{pmatrix}
1 & 1 & 1\\
0 & 1 & 2
\end{pmatrix}
$, and take $\w$ so that
$\ini_\w(I_A)=\langle \partial_1\partial_3\rangle$.
Hence
$$
\cS(\ini_\w(I_A))=
\{ (0,*,*),\, (*,*,0)\},
$$
where, for a standard pair $(\a,\sigma)$,
we put $*$ in the place of $\sigma$, $a_j$ at
$j\notin \sigma$.

Let $\bbeta=
\begin{pmatrix}
10\\
8
\end{pmatrix}
$.
Then the fake exponents are $(2, 8, 0)^T$ and $(0, 12, -2)^T$.
Since $(2, 8, 0)^T$ has minimal negative support,
$\phi_{(2, 8, 0)^T}$ is a solution.
\end{example}

\begin{example}
\label{running2-1}
Let 
$A=
\begin{pmatrix}
1 & 1 & 1 & 1\\
0 & 1 & 3 & 4
\end{pmatrix}
$.
Then
$$
I_A=\langle \,
\partial_1^2\partial_3-\partial_2^3,\,
\partial_2\partial_4^2-\partial_3^3,\,
\partial_1\partial_4-\partial_2\partial_3,\,
\partial_1\partial_3^2-\partial_2^2\partial_4\,
\rangle.
$$
Take $\w$ so that
$$
\ini_{\w}I_A
=\langle\, \partial_1^2\partial_3,\,
\partial_2\partial_4^2,\,
\partial_1\partial_4,\,
\partial_1\partial_3^2\,
\rangle.
$$
Hence
$$
\cS(\ini_\w(I_A))=
\{ (0,0,*, *),\,
(0, *, *, 0),\,
(0, *, *, 1),\,
(*, *, 0, 0),\,
(1, *, 1, 0)\}.
$$
Let $\bbeta=
\begin{pmatrix}
-1\\
-2
\end{pmatrix}
$.
Then the fake exponents (with their corresponding standard pairs) are
\begin{itemize}
\item
$\v_{3}:=(0,0,-7,5)^T \quad\leftrightarrow\quad (0,0,*,*)$,
\item
$\v_\emptyset:=(0,-5/2,1/2,0)^T \quad\leftrightarrow\quad (0,*,*,0)$,
\item
$\v_{2,3}:=(0,-2,-1,1)^T \quad\leftrightarrow\quad (0,*,*,1)$,
\item
$\v_{1,2}:=(-1,-1,0,0)^T \quad\leftrightarrow\quad (*,*,0,0)$,
\item
$\v_2:=(1,-4,1,0)^T \quad\leftrightarrow\quad (1,*,1,0)$.
\end{itemize}
Since $\v_\emptyset, \v_2, \v_{3}$ have minimal negative supports,
$\phi_{\v_\emptyset}, \phi_{\v_2}, \phi_{\v_{3}}$
are solutions.
\end{example}

\begin{example}[cf. Example 3.5.2 in \cite{SST}]
\label{running3-1}
Let 
$A=
\begin{pmatrix}
1 & 1 & 1 & 1 & 1\\
-1 & 1 & 1 & -1 & 0\\
-1 & -1 & 1 & 1 & 0
\end{pmatrix}
$. 
Then
$$
I_A=\langle \,
\partial_1\partial_3-\partial_5^2,\,
\partial_2\partial_4-\partial_5^2\,
\rangle.
$$
Take $\w$ so that
$\ini_\w(I_A)=
\langle\, \partial_1\partial_3,\, \partial_2\partial_4\,\rangle$.
Hence
$$
\cS(\ini_\w(I_A))=
\{ (0,0,*,*,*),\, (*,0,0,*,*),\, (*,*,0,0,*),\, (0,*,*,0,*)\}.
$$

Let
$\displaystyle
\bbeta=
\begin{pmatrix}
1\\
0\\
0
\end{pmatrix}.
$
Then $(0, 0, 0, 0,1)^T$ is a unique fake exponent.
Hence $\phi_{(0, 0, 0, 0,1)^T}=x_5$ is a solution.
\end{example}
\section{Negative supports and hyperplane arrangements}

In this section,
first we see that the sum \eqref{eq:SeriesInW} is
taken over a set of negative supports.
Then we define a set $\NS_\w(\v)$ of negative supports,
and we introduce a Gale dual and a hyperplane arrangement
 to visualize negative supports.

For an $A$-hypergeometric series $\phi$ \eqref{eq:SeriesInW},
set
$$
\supp(\phi):=\{ \u \, |\, g_\u\neq 0\}.
$$

\begin{proposition}
\label{lem:nsupp:equiv:class}
Let $\phi$ be an $A$-hypergeometric series \eqref{eq:SeriesInW}.
Suppose that $\u\in \supp(\phi)$ and
$\nsupp(\v+\u)=\nsupp(\v+\u')$ for $\u, \u'\in L$.

Then
$\u'\in \supp(\phi)$.
Furthermore, the highest log-terms of
$x^{\v+\u}$ and $x^{\v+\u'}$ are the same up to
nonzero scalar multiplication.
\end{proposition}

\begin{proof}
Suppose that
$x^{\v+\u}p(\log x)$
is the highest log-terms of $x^{\v+\u}$.
Consider
$$
(\u'-\u)=(\u'-\u)_+-(\u'-\u)_-
$$
and
$$
\partial^{(\u'-\u)_-}x^{\v+\u}p(\log x).
$$

Suppose that $u'_j-u_j<0$.
\begin{itemize}
\item
If $j\in \nsupp(\v+\u)=\nsupp(\v+\u')$,
then $v_j+u'_j=v_j+u_j+(u'_j-u_j)<v_j+u_j<0$.
Hence
$[v_j+u_j]_{u_j-u'_j}\neq 0$.
\item
Suppose that $j\notin \nsupp(\v+\u)=\nsupp(\v+\u')$.
If $v_j\notin \Z$, then clearly $[v_j+u_j]_{u_j-u'_j}\neq 0$.
If $v_j\in \Z$,
then $0<v_j+u'_j=v_j+u_j+(u'_j-u_j)<v_j+u_j$.
Hence
$[v_j+u_j]_{u_j-u'_j}\neq 0$.
\end{itemize}
Hence
the highest log-term of
$\partial^{(\u'-\u)_-}(x^{\v+\u}p(\log x))$
is equal to
$$
(\partial^{(\u'-\u)_-}x^{\v+\u})p(\log x),
$$
which is not zero.
Hence $\u'\in \supp(\phi)$.
Do the same argument exchanging $\u$ and $\u'$,
and find that
the highest log-terms of
$x^{\v+\u}$ and $x^{\v+\u'}$ are the same up to
nonzero scalar multiplication.
\end{proof}

Let $\w$ be a generic weight.
Suppose that
$\cG:=\{ {\partial^{\u^{(i)}_+}}-\partial^{\u^{(i)}_-}\,|\, i=1,2,\ldots, r \}$
is a Gr\"obner basis of $I_A$ with respect to $\w$, and that
$\partial^{\u^{(i)}_+}\in \ini_\w (I_A)$ for all $i$.
Set
$$
C(\w):=\sum_{i=1}^r \N \u^{(i)}.
$$

\begin{lemma}[cf. Theorem 6.12.14 in \cite{dojo}]
\label{lem:dojo}
Suppose that $\u\in L$ satisfies $\partial^{\u_+}\in \ini_\w (I_A)$.

Then
$\u\in C(\w)$.
\end{lemma}

\begin{proof}
Since 
$\cG$
is a Gr\"obner basis,
$\partial^{\u_+}-\partial^{\u_-}$
is reduced to $0$ by $\cG$.
This means that $\u$ belongs to $\sum_{i=1}^r \N \u^{(i)}$.
\end{proof}

\begin{corollary}
\label{cor:dojo}
Let $\phi$ be an $A$-hypergeometric series with
exponent $\v$ in the direction of $\w$.
Suppose that $\u\in \supp(\phi)$.
Then
$$
\{ \u'\in L\,|\, \nsupp(\v+\u')=\nsupp(\v+\u)\}
\subseteq 
C(\w).
$$
\end{corollary}

\begin{proof}
We may assume that $w_1,\ldots,w_n$ are linearly independent
over $\Q$.

By Proposition \ref{lem:nsupp:equiv:class},
for $\u'\in L$ with $\nsupp(\v+\u')=\nsupp(\v+\u)$,
we have $\u'\in \supp(\phi)$.
If $\u'=\0$, then clearly $\u'\in C(\w)$.

Let $\u'\neq \0$.
Since $\phi$ is in the direction of $\w$,
we have $\w\cdot \u'>0$, or $\w\cdot \u'_+ >\w\cdot \u'_-$.
By Lemma \ref{lem:dojo},
$\u'$ belongs to $C(\w)$.
\end{proof}

For a fake exponent $\v$, set
\begin{equation}
\NS_\w(\v):=
\left\{ I\,|\, 
\begin{array}{l}
\text{$I=\nsupp(\v+\u)$ for some $\u\in C(\w)$.}\\
\text{If $\nsupp(\v+\u')=
I$, then $\u'\in C(\w)$.}
\end{array}
\right\},
\end{equation}
and
$$
\NS_\w(\v)^c:=\{ \nsupp(\v+\u)\,|\, \u\in L\}\setminus \NS_\w(\v).
$$

\begin{proposition}
\label{cor:nsupp(v)inNS}
Let $\w$ be a generic weight, and $\v$ a fake exponent.
Then

$$
\{ \nsupp(\v+\u)\,|\, \nsupp(\v+\u)\subseteq \nsupp(\v),\, \u\in L\}
\subseteq \NS_\w(\v).
$$
In particular, $\nsupp(\v)\in \NS_\w(\v)$.
\end{proposition}

\begin{proof}
We may assume that $w_1,\ldots,w_n$ are linearly independent
over $\Q$.

Let $\u\in L\setminus \{ \0\}$. 
We show 
\begin{equation}
\label{eqn:Prop4.3}
\nsupp(\v+\u)\subseteq \nsupp(\v) \Rightarrow \w\cdot \u>0.
\end{equation}
Suppose that $\w\cdot\u<0$. Then
$\partial^{\u_-} x^\v=0$ since $\v$ is a fake exponent.
Hence there exists $j$ such that $u_j<0$, $v_j\in \N$, and $v_j+u_j<0$.
Namely, $j\in \nsupp(\v+\u)\setminus \nsupp(\v)$,
and we have proved \eqref{eqn:Prop4.3}.
Then by Lemma \ref{lem:dojo}, 
we have $\u\in C(\w)$
(this is also valid for $\u=\0$.).
\end{proof}

\begin{corollary}
\label{cor:M>=m}
Let $\v$ have the smallest $\w$-weight among the set of
fake exponents in $\v+L$.

If $\v+\u_0$ is a fake exponent,
then $\nsupp(\v+\u_0)\in \NS_\w(\v)$.
\end{corollary}

\begin{proof}
We may assume that $w_1,\ldots,w_n$ are linearly independent
over $\Q$.

If $\u_0=\0$, then the assertion is in Proposition \ref{cor:nsupp(v)inNS}.
Suppose that $\u_0\neq \0$.
By the minimality of $\v$, we have $\w\cdot \u_0>0$.
If $\nsupp(\v+\u_0+\u)=\nsupp(\v+\u_0)$,
then $\w\cdot \u\geq 0$ by \eqref{eqn:Prop4.3}.
Hence $\u_0+\u\in C(\w)$ by Lemma \ref{lem:dojo}, and
$\nsupp(\v+\u_0)\in \NS_\w(\v)$.
\end{proof}

\bigskip
To visualize $\NS_\w(\v)$, we introduce a Gale
dual (cf. e.g. \cite{Zieg}).
Let
$\{ \b_1,\b_2, \ldots, \b_{n-d}\}$ be a basis of $L$.
Set
$$
B:=
(\b_1, \b_2,\ldots, \b_{n-d})=(\g_1,\g_2,\ldots,\g_n)^T.
$$
For $\v\in \C^n$, define
$$
\psi_\v: \R^{n-d}\simeq \v+L_{\R}\subseteq \C^n
$$
by $\psi_\v(\x)=\v+B\x$,
where $L_\R=\R\otimes_\Z L$.
Set $Z(\v):=\{ i\in [1,n]\,|\, v_i\in \Z\}$.
For a subset $I\subseteq Z(\v)$, set
\begin{equation}
N_I(\v):=\{ \v+\u\,|\, \nsupp(\v+\u)=I\}.
\end{equation}
Then
\begin{eqnarray*}
N_I(\v)
&=&
\{ \z\in \v+L\,|\,
z_i< 0\, (i\in I);\,\, z_j\geq 0\, (j\in Z(\v)\setminus I)
\}\\
&=&
\{ \z\in \v+L\,|\,
\e^*_i(\z)< 0\, (i\in I);\,\, \e^*_j(\z)\geq 0\, (j\in Z(\v)\setminus I)
\},
\end{eqnarray*}
where $\{ \e^*_i \,|\, 1\leq i\leq n\}$ is the dual basis of the
standard basis $\{ \e_i \,|\, 1\leq i\leq n\}$ of $\C^n$.
Hence $N_I(\v)$ is the set of lattice points in a union of
faces of the hyperplane arrangement $\{ \e^*_i=0\, |\, i\in Z(\v)\}$
on $\v+L_\R$.
Transfer this to the hyperplane arrangement on $\R^{n-d}$
by $\psi_\v$. Since
\begin{eqnarray*}
(\psi_\v^*(\e^*_i))(\x)
&=&
(\e^*_i)(\psi_\v(\x))=(\e^*_i)(\v +B\x)\\
&=& v_i+(B^T \e_i)(\x)
=v_i+\g_i(\x),
\end{eqnarray*}
we can regard $N_I(\v)$ as the set of lattice points in a union of faces
of the hyperplane arrangement $\{ H_i\,|\, i\in Z(\v)\}$
on $\R^{n-d}$, where
$$
H_i=\{ \x\in \R^{n-d}\,|\, \g_i(\x)+v_i=0\}.
$$

\begin{example}[Continuation of Example \ref{running1-1}]
\label{running1-2}
Let 
$A=
\begin{pmatrix}
1 & 1 & 1\\
0 & 1 & 2
\end{pmatrix}
$, and take $\w$ as before.
Then
$C(\w)=\N(1,-2,1)$.
Let
$$
B=
\begin{pmatrix}
1\\
-2\\
1
\end{pmatrix}
=(\b)=(g_1,g_2,g_3)^T.
$$

Let $\bbeta=
\begin{pmatrix}
10\\
8
\end{pmatrix}
$.
Then the fake exponents are $\v:=(0, 12, -2)^T$ and $\v':=(2, 8, 0)^T$.
We have
$$
\psi_{\v}:\R\ni x\mapsto
\v+x\b=(x,-2x+12,x-2)^T\in \R^3,
$$
and
$$
\nsupp(\v+x\b)=
\left\{
\begin{array}{ll}
\{ 2\} & (x\geq 7)\\
\emptyset & (x=2,3,4,5,6)\\
\{ 3\} & (x=0,1)\\
\{ 1,3\} & (x\leq -1).
\end{array}
\right.
$$
Hence
$$
\NS_\w(\v)=\{ \{ 2\}, \{ 3\}, \emptyset\},\qquad
\NS_\w(\v)^c=\{ \{ 1,3\}\}.
$$
In the following picture, a small arrow indicates the positive side.
Note that a hyperplane (a point in this example) itself belongs 
to its positive side.
\begin{center}
\begin{tikzpicture}
\draw[->] (-5.2,0) -- (5.2,0) node[right] {$x$};
\draw[color=blue] (-3,-0.8) -- (-3,0.8) node[above] {$H_1$};
\draw[color=blue, ->] (-3,0.4) -- (-2.8,0.4);
\draw[color=blue] (-1,-0.8) -- (-1,0.8) node[above] {$H_3$};
\draw[color=blue, ->] (-1,0.4) -- (-0.8,0.4);
\draw[color=blue] (3,-0.8) -- (3,0.8) node[above] {$H_2$};
\draw[color=blue, ->] (3,0.4) -- (2.8,0.4);
\coordinate (-2) at (-5,0) node at (-2) [below] {$-2$}; 
\coordinate (-1) at (-4,0) node at (-1) [below] {$-1$}; 
\coordinate (A) at (-3,0) node at (A) [below left] {$\v$}; 
\coordinate (1) at (-2,0) node at (1) [below] {$1$}; 
\coordinate (B) at (-1,0) node at (B) [below left] {$\v'$}; 
\coordinate (3) at (0,0) node at (3) [below] {$3$}; 
\coordinate (4) at (1,0) node at (4) [below] {$4$}; 
\coordinate (5) at (2,0) node at (5) [below] {$5$}; 
\coordinate (7) at (4,0) node at (7) [below] {$7$}; 
\node (C) at (-4.5, 0.5) {$\{ 1,3\}$};
\node (D) at (-2, 0.5) {$\{ 3\}$};
\node (E) at (1, 0.5) {$\emptyset$};
\node (F) at (4, 0.5) {$\{ 2\}$};
\fill (A) circle (2pt) (B) circle (2pt) (-2) circle (1pt) (-1) circle (1pt) (1) circle (1pt) (3) circle (1pt) (4) circle (1pt) (5) circle (1pt) (7) circle (1pt); 
\end{tikzpicture}
\end{center}
\end{example}

\begin{example}[Continuation of Example \ref{running2-1}]
\label{running2-2}
Let 
$A=
\begin{pmatrix}
1 & 1 & 1 & 1\\
0 & 1 & 3 & 4
\end{pmatrix}
$, and take $\w$ as before.
Then
\begin{eqnarray*}
C(\w)&=&\N(1,-1,-1,1)^T+ \N(2,-3,1,0)^T+
\N(0,1,-3,2)^T+ \N(1,-2,2,-1)^T\\
&=&\N(1,-2,2,-1)^T\oplus \N(0,1,-3,2)^T.
\end{eqnarray*}
Let
$$
B:=
\begin{pmatrix}
1 & 0\\
-2 & 1\\
2 & -3\\
-1 & 2
\end{pmatrix}
=(\b_1,\b_2)
=(\g_1, \g_2, \g_3, \g_4)^T.
$$

Let $\bbeta=
\begin{pmatrix}
-2\\
-1
\end{pmatrix}
$, and
$\v:=\v_{1,2}=(-1,-1,0,0)^T$.
Then
$$
\NS_\w(\v)=\{ \{ 2\},\{ 3\}, \{ 2,3\}, \{ 1,2\}=I_\0\},
$$
$$
\NS_\w(\v)^c=\{ \{ 1,3\},\{ 2,4\}, \{ 1,4\},
\{ 1,3,4\}, \{ 1,2,4\}\}.
$$
In the following picture, we put $I$ in the face
where $\nsupp(\v+x_1\b_1+x_2\b_2)=I$ for a lattice point
$(x_1,x_2)^T$.

\begin{center}
\begin{tikzpicture}
\draw[very thin,color=gray] (-4.5,-4.5) grid (4.5,4.5);
\draw[->] (-5.2,0) -- (5.2,0) node[right] {$x_1$};
\draw[->] (0,-5.2) -- (0,5.2) node[above] {$x_2$};
\draw[color=blue] (1,-5.2) -- (1,5.2) node[above] {$H_1$};
\draw[color=blue] (-3,-5) -- (2,5) node[above] {$H_2$};
\draw[color=blue] plot (\x,0.66*\x) node[right] {$H_3$};
\draw[color=blue] plot (\x,0.5*\x) node[right] {$H_4$};
\coordinate (A) at (0,0) node at (A) [above left] {$\v$}; 

\fill (A) circle (2pt); 
\draw[color=blue, ->] (1,3.8) -- (1.2,3.8);
\draw[color=blue, ->] (0.5, 2) -- (0.3,2.1);
\draw[color=blue, ->] (3,2) -- (3.1, 1.85);
\draw[color=blue, ->] (3,1.5) -- (2.9, 1.7);
\node (E) at (-1.5, 1.5) {$\{ 1,3\}$};
\node (F) at (1.5, 4.7) {$\{ 3\}$};
\node (G) at (3.5, 3.5) {$\{ 2,3\}$};
\node (H) at (4, 2.3) {$\{ 2\}$};
\node (I) at (3.5, 0.5) {$\{ 2,4\}$};
\node (J) at (-0.5, -2.5) {$\{ 1,2,4\}$};
\node (K) at (-2.5, -2.5) {$\{ 1,4\}$};
\node (L) at (-5, -2.8) {$\{ 1,3,4\}$};
\end{tikzpicture}
\end{center}
\end{example}

\begin{example}[Continuation of Example \ref{running3-1}]
\label{running3-2}

\smallskip

Let 
$A=
\begin{pmatrix}
1 & 1 & 1 & 1 & 1\\
-1 & 1 & 1 & -1 & 0\\
-1 & -1 & 1 & 1 & 0
\end{pmatrix}
$, and take $\w$ as before.
Let
$$
B=
\begin{pmatrix}
1 & 0\\
0 & 1\\
1 & 0\\
0 & 1\\
-2 & -2
\end{pmatrix}
=(\b_1,\b_2)
=(\g_1,\ldots,\g_5)^T.
$$

Let
$\displaystyle
\bbeta=
\begin{pmatrix}
1\\
0\\
0
\end{pmatrix}.
$
Then $\v:=(0, 0, 0, 0,1)^T$ is the unique fake exponent.

$$
\NS_\w(\v)=\{ \emptyset, \{ 5\}\},
$$
$$
\NS_\w(\v)^c=\{ \{ 1,3\}, \{ 2,4\}, \{ 1,3,5\}, 
\{ 2,4,5\}, \{ 1,2,3,4\}\}.
$$
\begin{center}
\begin{tikzpicture}
\draw[very thin,color=gray] (-3.5,-3.5) grid (3.5,3.5);
\draw[color=blue, ->] (-4.2,0) -- (4.2,0) node[above] {$H_1=H_3$};
\draw[color=blue, ->] (0,-4.2) -- (0,4.2) node[right] {$H_2=H_4$};
\draw[color=blue] (-3.5, 4) -- (4, -3.5) node[right] {$H_5$};
\coordinate (A) at (0,0) node at (A) [above left] {$\v$}; 

\fill (A) circle (2pt); 
\draw[color=blue, ->] (1.5,0) -- (1.5,0.2);
\draw[color=blue, ->] (0, 1.5) -- (0.2,1.5);
\draw[color=blue, ->] (-0.5,1) -- (-0.65, 0.85);
\node (C) at (4.5, 0) {$x_1$};
\node (D) at (0, 4.5) {$x_2$};
\node (E) at (2.5, 2.5) {$\{ 5\}$};
\node (F) at (3.5, -1.5) {$\{ 1,3,5\}$};
\node (G) at (1.5, -2.5) {$\{ 1,3\}$};
\node (H) at (-2.5, -2.5) {$\{ 1,2,3,4\}$};
\node (I) at (-2.5, 1.5) {$\{ 2,4\}$};
\node (J) at (-1.5, 3.5) {$\{ 2,4,5\}$};
\end{tikzpicture}
\end{center}
\end{example}
\section{Method 1}

In this section, we consider a Frobenius's method
by perturbing an exponent with a single vector in $L$.

\begin{lemma}
\label{CoefOrder}
Let $\b\in L$, $\u\in \N^n$.
Then
\begin{eqnarray*}
[\v+s\b]_\u
&=&
(\!\!\!\!\!\!\!\!\!\!\!\!\!\!\!
\prod_{i\in \nsupp(\v-\u)\setminus\nsupp(\v)}\!\!\!\!\!\!\!
\!\!\!\!\!\!\!\!\!\!
b_i)
\widehat{[\v]}_{\u}s^{|\nsupp(\v-\u)|-|\nsupp(\v)|}\\
&&\qquad +o(s^{|\nsupp(\v-\u)|-|\nsupp(\v)|}),
\end{eqnarray*}
where $\widehat{[\v]}_\u$ is the product of nonzero factors of
$[\v]_\u$;
\begin{eqnarray*}
{\widehat{[\v]}}_{\u}
&=&(\prod_{
i\notin \nsupp(\v-\u)\setminus\nsupp(\v)}[v_i]_{u_i})\\
&&\quad \times
(\prod_{
i\in \nsupp(\v-\u)\setminus\nsupp(\v)}(v_i)!(-1)^{|v_i-u_i+1|}(|v_i-u_i+1|!)).
\end{eqnarray*}
\end{lemma}

\begin{proof}
Note that
$$
[\v+s\b]_\u=\prod_{i=1}^n
(v_i+sb_i)(v_i-1+sb_i)\cdots(v_i-u_i+1+sb_i).
$$
In $(v_i+sb_i)(v_i-1+sb_i)\cdots(v_i-u_i+1+sb_i)$,
the factor $s$ appears if and only if
$v_i\in \N$ and $v_i-u_i\in \Z_{<0}$, and
furthermore if that is the case, it appears only once
and always with $b_i$.

Finally note that $\nsupp(\v-\u)\supseteq \nsupp(\v)$.
Hence
$|\nsupp(\v-\u)\setminus\nsupp(\v)|=
|\nsupp(\v-\u)|-|\nsupp(\v)|$.
\end{proof}

\begin{corollary}
\label{cor:CoefOrder}
Let $\b,\u \in L$.
Suppose that $b_j\neq 0$ for any $j$. Then
\begin{enumerate}
\item
$\ord_s([\v+s\b]_{\u_-})
=|\nsupp(\v+\u)\setminus \nsupp(\v)|$.
\item
$\ord_s([\v+s\b+\u]_{\u_+})
=|\nsupp(\v)\setminus \nsupp(\v+\u)|$.
\end{enumerate}
\end{corollary}

\begin{proof}
(1)\quad
By Lemma \ref{CoefOrder},
$\ord_s([\v+s\b]_{\u_-})
=|\nsupp(\v-\u_-)\setminus \nsupp(\v)|$.

We have $\nsupp(\v+\u)=\nsupp(\v+\u_+-\u_-)
\subseteq \nsupp(\v-\u_-)$.
Hence $(\nsupp(\v+\u)\setminus \nsupp(\v))\subseteq
\nsupp(\v-\u_-)\setminus \nsupp(\v)$.

Let $j\in \nsupp(\v-\u_-)\setminus \nsupp(\v)$.
Then $u_j<0$. Hence $(\u_+)_j=0$, and
$j\in \nsupp(\v-\u_-+\u_+)\setminus \nsupp(\v)$.
Therefore $\nsupp(\v-\u_-)\setminus \nsupp(\v)
=\nsupp(\v+\u)\setminus \nsupp(\v)$.

(2)\quad
In (1), replace $\v, \u$ by
$\v+\u, -\u$, respectively.
\end{proof}

\begin{corollary}
\label{cor:a_u:CoefOrder}
Let $\b,\u \in L$.
Suppose that $b_j\neq 0$ for any $j$. 
Let
$$
a_\u(s):=\frac{[\v+s\b]_{\u_-}}{[\v+s\b+\u]_{\u_+}}.
$$
Then
$$
\ord_s(a_\u(s))
=|\nsupp(\v+\u)|- |\nsupp(\v)|.
$$
Indeed,
\begin{eqnarray*}
a_\u(s)
&=&
\frac{\prod_{i\in \nsupp(\v+\u)\setminus\nsupp(\v)}b_i}{\prod_{j\in 
\nsupp(\v)\setminus\nsupp(\v+\u)}b_j}
\frac{{\widehat{[\v]}}_{\u_-}}{{\widehat{[\v+\u]}}_{\u_+}}
s^{|\nsupp(\v+\u)|-|\nsupp(\v)|}\\
&&\qquad\qquad\qquad +o(s^{|\nsupp(\v+\u)|-|\nsupp(\v)|}).
\end{eqnarray*}
\end{corollary}

\begin{proof}
For finite sets $X$ and $Y$,
\begin{eqnarray*}
|X\setminus Y|-|Y\setminus X|
&=&
|X\setminus Y|+|X\cap Y|-(|Y\setminus X|+|X\cap Y|)\\
&=&|X|-|Y|.
\end{eqnarray*}
Hence the statement follows from Lemma \ref{CoefOrder}
and Corollary \ref{cor:CoefOrder}.
\end{proof}

\bigskip

Let $\v$ be a fake exponent.
For $\b\in L$ with $b_i\neq 0$ ($i\in \nsupp(\v)$), set
$$
F_\b(x,s):=
\sum_{\nsupp(\v+\u)\in \NS_\w(\v)}
a_\u(s) x^{\v+s\b+\u},
$$
where
$$
a_\u(s)=\frac{[\v+s\b]_{\u_-}}{[\v+s\b+\u]_{\u_+}}.
$$
The condition $b_i\neq 0$ ($i\in \nsupp(\v)$)
guarantees the denominator of $a_\u(s)$ not to be zero
by Corollary \ref{cor:a_u:CoefOrder}.

Set $I_\u:=\nsupp(\v+\u)$ for $\u\in L$.

\begin{theorem}
\label{thm:1}
Let $\v$ be a fake exponent.
Put
$m:=\min_{I\in \NS_\w(\v)}|I|$ and 
$M:=\min_{I\in \NS_\w(\v), J\in \NS_\w(\v)^c}(|I\cup J|)$.
Since
$I_\0\in \NS_\w(\v)$ (Proposition \ref{cor:nsupp(v)inNS}), 
we have $|I_\0|\geq m$.
Let $\b\in L$ satisfy $b_i\neq 0$ for $i\in \nsupp(\v)$. Then

\begin{enumerate}
\item[\rm (1)]
$
(\partial_s^j s^{|I_\0|-m}F_\b(x,s))_{|s=0}\quad (j=0,1,\ldots, M-m-1)
$
are solutions to $M_A(\bbeta)$.

If $M >|I_\0|$, then $\v$ is an exponent with multiplicity
 at least
$\binom{n-d+M-|I_\0|-1}{M-|I_\0|-1}$.

\item[\rm (2)]
If
$\b^{(1)}, \b^{(2)},\ldots \b^{(k)}\in L$
satisfy
$$
\sum_{i=1}^{k}
\frac{b^{(i)}_{I\setminus I_\0}b^{(i)}_{J\setminus I}}{b^{(i)}_{I_\0\setminus I}}=0
$$
for all $I\in \NS_\w(\v), J\in \NS_\w(\v)^c$ with 
$|I\cup J|=M$,
then
$$
(\partial_s^{M-m} s^{|I_\0|-m}\sum_{i=1}^kF_{\b^{(i)}}(x,s))_{|s=0}
$$
is also a solution to $M_A(\bbeta)$,
where
$b_{K}=\prod_{k\in K}b_k$.
If $M\geq |I_\0|$, then $\v$ is an exponent.
\end{enumerate}
\end{theorem}

\begin{proof}
First of all, since $\b\in L$,
we have
$$
(\sum_{j=1}^n a_{ij}\theta_{j} -\beta_i)
F_\b(x,s)=0
\qquad
(i=1,\ldots, d).
$$

Let $\u'\in L$.

Suppose that
$I_\u, I_{\u+\u'}\in \NS_\w(\v)$.
Then as in \cite[(3.29)]{SST}
$$
\partial^{\u'_-}(a_\u(s)x^{\v+s\b+\u})
=\partial^{\u'_+}(a_{\u+\u'}(s)x^{\v+s\b+\u+\u'}).
$$
Hence
\begin{eqnarray*}
&&
(\partial^{\u'_+}-\partial^{\u'_-})
F_\b(x,s)\\
&=&
\sum_{I_\u\in \NS_\w(\v),\,
I_{\u-\u'}\in \NS_\w(\v)^c}
\partial^{\u'_+}(a_\u(s)x^{\v+s\b+\u})\\
&&\quad -
\sum_{I_\u\in \NS_\w(\v),\,
I_{\u+\u'}\in \NS_\w(\v)^c}
\partial^{\u'_-}(a_\u(s)x^{\v+s\b+\u})\\
&=&
\sum_{I_\u\in \NS_\w(\v),\,
I_{\u-\u'}\in \NS_\w(\v)^c}
\partial^{\u'_+}(a_\u(s)x^{\v+s\b+\u})\\
&&\quad -
\sum_{I_\u\in \NS_\w(\v),\,
I_{\u-(-\u')}\in \NS_\w(\v)^c}
\partial^{(-\u')_+}(a_\u(s)x^{\v+s\b+\u}).
\end{eqnarray*}

Suppose that $I_\u\in \NS_\w(\v)$ and
$J:=I_{\u-\u'}\in \NS_\w(\v)^c$.
Then
$u'_j>0\, (j\in J\setminus I_\u)$, and
$$
\partial^{\u'_+}(a_\u(s)x^{\v+s\b+\u})
=
a_\u(s)[\v+s\b+\u]_{\u'_+}x^{\v+s\b+\u-\u'_+}.
$$
By Corollary \ref{cor:a_u:CoefOrder},
$$
a_\u(s)=
c\frac{b_{I_\u\setminus I_\0}}{b_{I_\0\setminus I_\u}}s^{|I_\u|-|I_\0|}
+\text{higher terms},
$$
where
$c$ is a  nonzero constant unrelated to $s$ and $b_j$'s.
Hence $s^{|I_\0|-m}a_\u(s)$ does not have a pole at $s=0$.

By Lemma \ref{CoefOrder},
$$
[\v+s\b+\u]_{\u'_+}=c' b_{J\setminus I_\u}s^{|J\setminus I_\u|}+
\text{higher terms}.
$$
Hence, 
\begin{eqnarray}
\partial^{\u'_+}(a_\u(s)x^{\v+s\b+\u})
&=&
c\frac{b_{I_\u\setminus I_\0}b_{J\setminus I_\u}}{b_{I_\0\setminus I_\u}}
s^{|I_\u|-|I_\0|}s^{|J\setminus I_\u|}
+\text{higher terms}\nonumber\\
&=&
c\frac{b_{I_\u\setminus I_\0}b_{J\setminus I_\u}}{b_{I_\0\setminus I_\u}}
s^{|I_\u\cup J|-|I_\0|}
+\text{higher terms}.
\label{eqn:ExternalArea2}
\end{eqnarray}

Thus each coefficient of 
$\partial^{\u'_+}s^{|I_\0|-m}(a_\u(s)x^{\v+s\b+\u})$
has order at least $M-m$ in $s$, and we have proved the first part of (1).
By looking at the coefficient of \eqref{eqn:ExternalArea2}, 
we have (2).

Note that the starting part of 
$(\partial_s^{|I_\0|-m+k} s^{|I_\0|-m}F_\b(x,s))_{|s=0}$ is a 
nonzero multiple of $x^\v(\log x^\b)^k$
$(k=0,1,\ldots, M-|I_\0|-1)$.
Since $\rank\, L=n-d$, we have the second part of (1).
\end{proof}

\begin{remark}
Since the degrees of logarithmic polynomials in the coefficients of
$(\partial_s^j s^{|I_\0|-m}F_\b(x,s))_{|s=0}$ are less than or equal to 
$j$,
$$
(\partial_s^j s^{|I_\0|-m}F_\b(x,s))_{|s=0}\qquad
(j= |I_\0|-m, \ldots, M-m-1)
$$
in Theorem \ref{thm:1} (1) are basic Nilsson series solutions
\cite[Definition 2.6]{DMM}.
\end{remark}

\begin{remark}
\label{rmk:Method1}
In Theorem \ref{thm:1}, we may replace $\NS_\w(\v)$
and $\NS_\w(\v)^c$ by any $N\subseteq \NS_\w(\v)$ with $N\ni\nsupp(\v)$
and $N^c:=(\NS_\w(\v)\cup \NS_\w(\v)^c)\setminus N$.
Indeed, the proof of Theorem \ref{thm:1} is also valid
for $N$ and $N^c$.
\end{remark}

\begin{example}
Let $\v$ be a fake exponent with minimal negative support,
and $N:=\{ \nsupp(\v)\}$.

Then for any $J\in N^c$ we have $I_\0\cup J\supsetneq I_\0$.
Hence $M>|I_\0|=m$, and we see that
$\phi_\v$ is a solution with exponent $\v$
by Theorem \ref{thm:1} for $N$.
\end{example}

\begin{example}[Continuation of Example \ref{running1-2}]
\label{running1-3}
Let 
$A=
\begin{pmatrix}
1 & 1 & 1\\
0 & 1 & 2
\end{pmatrix}
$, and take $\w$ as before.

Let $\bbeta=
\begin{pmatrix}
10\\
8
\end{pmatrix}
$, and
$\v:=(0, 12, -2)^T$.

Then
$$
\NS_\w(\v)=\{ \{ 2\}, \{ 3\}, \emptyset\},\qquad
\NS_\w(\v)^c=\{ \{ 1,3\}\}.
$$
Hence
$$
M=2,\quad m=0,\quad I_0=\{ 3\},
$$
and by Theorem \ref{thm:1}
$(\partial^j_s sF_\b(x,s))_{|s=0}$ $(j=0,1)$ are solutions.
Here
$$
(\partial^0_s sF_\b(x,s))_{|s=0}
=c \phi_{\v'},
$$
where $\b=(1,-2,1)^T$, $\v'=(2,8,0)^T=\v+2\b$, and 
$$
c=(sa_{2\b}(s))_{|s=0}=(\frac{s[12-2s]_4}{[s+2]_2[s]_2})_{|s=0}
=-5940.
$$
$$
(\partial_s sF_\b(x,s))_{|s=0}=
\sum_{k\neq 2,3,\ldots,6}a_{k\b}(0)x^{\v+k\b}
+\sum_{k=2}^6(sa_{k\b}(s))_{|s=0}(\log x^\b) x^{\v+k\b}.
$$
Here note that $a_{k\b}(s)$ has a pole of order $1$ at $s=0$
for $k=2,3,4,5,6$
by Corollary \ref{cor:a_u:CoefOrder}.

Note that the $\psi(0,x)$ in \cite[Example 3.5.3]{SST} has
a typo.
\end{example}

\begin{example}[Continuation of Example \ref{running2-2}]
\label{running2-3}
Let 
$A=
\begin{pmatrix}
1 & 1 & 1 & 1\\
0 & 1 & 3 & 4
\end{pmatrix}
$, and take $\w$ as before.

Let $\bbeta=
\begin{pmatrix}
-2\\
-1
\end{pmatrix}
$,
and $\v:=\v_{1,2}=(-1,-1,0,0)^T$.
Then
$$
\NS_\w(\v)=\{ \{ 2\},\{ 3\}, \{ 2,3\}, \{ 1,2\}=I_\0\},
$$
$$
\NS_\w(\v)^c=\{ \{ 1,3\},\{ 2,4\}, \{ 1,4\},
\{ 1,3,4\}, \{ 1,2,4\}\}.
$$
Hence
$M=2, m=1, I_\0=\{ 1,2\}$.
By Theorem \ref{thm:1} (1),
$(sF_\b(x,s))_{s=0}$ is a solution for any $\b\in L$ with $b_1,b_2\neq 0$.
We have $\v_2=\v+(2,-3,1,0)^T$ and
$\v_3=\v+(1,1,-7,5)^T$.
Hence by Corollary \ref{cor:a_u:CoefOrder}
\begin{eqnarray*}
(sF_\b(x,s))_{s=0}
&=&
(sa_{(2,-3,1,0)^T}(s))_{s=0}\phi_{\v_2}+
(sa_{(1,1,-7,5)^T}(s))_{s=0}\phi_{\v_3}\\
&=&
-\frac{6}{b_1}\phi_{\v_2}
+\frac{6b_3}{b_1b_2}\phi_{\v_3}.
\end{eqnarray*}

The $(I,J)$'s with the condition of Theorem \ref{thm:1} (2) are
$$
( \{ 3\}, \{ 1,3\}),\quad (\{ 2\}, \{ 2,4\}).
$$
Hence
if $\b^{(1)}, \b^{(2)}\in L$ satisfy
$$
\sum_{i=1}^2 \frac{b^{(i)}_3b^{(i)}_1}{b^{(i)}_1b^{(i)}_2}=
\sum_{i=1}^2 \frac{b^{(i)}_3}{b^{(i)}_2}=
0,\quad
\sum_{i=1}^2 \frac{b^{(i)}_4b^{(i)}_1}{b^{(i)}_1}=
\sum_{i=1}^2 b^{(i)}_4
=0,
$$
then
$(\partial_s\sum_{i=1}^2 sF_{\b^{(i)}})_{|s=0}$
is a solution.
We see that
$$
\b^{(1)}=(1,-2,2,-1)^T, \quad \b^{(2)}=(1,-1,-1,1)^T
$$
would do.
Hence $\v_{1,2}$ is an exponent, and
\begin{eqnarray*}
&&
(\partial_s \sum_{i=1}^2sF_{\b^{(i)}}(x,s))_{|s=0}\\
&=&
(\partial_s(\sum_{i=1}^2 \sum_{\nsupp(\v+\u)\in \NS_\w(\v)}
sa_\u^{(i)}(s)x^{\v+s\b^{(i)}+\u})_{|s=0}\\
&=&
\sum_{\nsupp(\v+\u)=\{ 1,2\}, \{ 2,3\}}
(a_\u^{(1)}(0)+a_\u^{(2)}(0))x^{\v+\u}\\
&&\quad+
\sum_{\nsupp(\v+\u)=\{ 2\}, \{ 3\}}
(\partial_s(sa_\u^{(1)}(s)+sa_\u^{(2)}(s)))_{|s=0}x^{\v+\u}\\
&&\quad
+\sum_{\nsupp(\v+\u)=\{ 2\}, \{ 3\}}
(sa_\u^{(1)}(s))_{|s=0}(\log x^{\b^{(1)}}) x^{\v+\u}\\
&&\quad
+\sum_{\nsupp(\v+\u)=\{ 2\}, \{ 3\}}
(sa_\u^{(2)}(s))_{|s=0}(\log x^{\b^{(2)}}) x^{\v+\u},
\end{eqnarray*}
where
$$
a_\u^{(i)}(s)=
\frac{[\v+s\b^{(i)}]_{\u_-}}{[\v+s\b^{(i)}+\u]_{\u_+}}.
$$
\end{example}

\section{Method 2}

In this section, we consider a Frobenius's method
by perturbing an exponent with several vectors in $L$.

Let $\v$ be a fake exponent, and
$\b^{(1)}, \b^{(2)},\ldots, \b^{(l)}\in L$.
We suppose that for any $i\in \nsupp(\v)$
there exists $j$ such that $b^{(j)}_i\neq 0$.

For such
$\b^{(1)}, \b^{(2)},\ldots, \b^{(l)}\in L$,
set
$$
F_{\b^{(1)}, \b^{(2)},\ldots, \b^{(l)}}(x,\s)
:=\sum_{\nsupp(\v+\u)\in \NS_\w(\v)}
a_\u(\s)x^{\v+\s\b+\u},
$$
where
$$
\s=(s_1,s_2,\ldots, s_l),\quad
\s\b=s_1\b^{(1)}+s_2\b^{(2)}+\cdots+s_l\b^{(l)},
$$
$$
a_\u(\s)=
\frac{[\v+\s\b]_{\u_-}}{[\v+\s\b+\u]_{\u_+}}.
$$

Similarly to Lemma \ref{CoefOrder} and Corollary \ref{cor:a_u:CoefOrder},
we have the following Lemma.

\begin{lemma}
\label{lem:a_u:multi}
\begin{enumerate}
\item[\rm (1)]
$$
a_\u(\s)=c
\frac{\prod_{i\in \nsupp(\v+\u)\setminus\nsupp(\v)}
(s_1b^{(1)}_i+s_2b^{(2)}_i+\cdots+s_lb^{(l)}_i)+\text{higher terms}}
{\prod_{j\in \nsupp(\v)\setminus\nsupp(\v+\u)}
(s_1b^{(1)}_j+s_2b^{(2)}_j+\cdots+s_lb^{(l)}_j)+\text{higher terms}}.
$$
\item[\rm (2)]
$$
[\v+\s\b+\u]_{\u'_+}
=c\!\!\!\!\!\!\!\!\!\!\!\!\!\!\!\!\!\!\!\!
\prod_{i\in \nsupp(\v+\u-\u')\setminus\nsupp(\v+\u)}
\!\!\!\!\!\!\!\!\!\!\!\!\!\!\!\!\!\!\!\!
(s_1b^{(1)}_i+s_2b^{(2)}_i+\cdots+s_lb^{(l)}_i)
+\text{higher terms}.
$$
\end{enumerate}
Here $c$ is a nonzero constant unrelated to $\s$ and
$\b^{(k)}$'s.
\end{lemma}

\begin{theorem}
\label{thm2}
Put
$K:=\cap_{I\in \NS_\w(\v)}I$.
Set
$$
\widetilde{F}(x,\s)
:=\prod_{i\in I_\0\setminus K}(s_1b^{(1)}_i+s_2b^{(2)}_i+\cdots+s_lb^{(l)}_i)
\cdot F_{\b^{(1)}, \b^{(2)},\ldots, \b^{(l)}}(x,\s).
$$
Let $M$ be the one in Theorem \ref{thm:1}.
Then
\begin{enumerate}
\item[\rm (1)]
$(\partial_{s_1}^{p_1}\cdots \partial_{s_l}^{p_l}
\widetilde{F}(x,\s))_{|\s=\0}$
are solutions to $M_A(\bbeta)$ for
$\sum_{k=1}^lp_k < M-|K|$.
\item[\rm (2)]
Suppose that $\sum_{k=1}^lp_k = M-|K|$.

If
$$
\sum_{\coprod_{j=1}^l L_j
=I\cup J\setminus K;\, |L_1|=p_1,\ldots, |L_l|=p_l}
\prod_{j=1}^lb^{(j)}_{L_j}
=0
$$
for all $I\in \NS_\w(\v)$ and $J\in \NS_\w(\v)^c$ with
$|I\cup J|=M$, then
$(\partial_{s_1}^{p_1}\cdots \partial_{s_l}^{p_l}
\widetilde{F}(x,\s))_{|\s=\0}$
is also a solution to $M_A(\bbeta)$.
\end{enumerate}
\end{theorem}

\begin{proof}
By Lemma \ref{lem:a_u:multi},
we see that 
$\prod_{i\in I_\0\setminus K}(s_1b^{(1)}_i+s_2b^{(2)}_i+\cdots+s_lb^{(l)}_i)
\cdot a_\u(\s)$
does not have a pole at $\s=\0$ for any $\u$ with $I_\u\in \NS_\w(\v)$.

Let $\u'\in L$, $\nsupp(\v+\u)=I\in \NS_\w(\v)$, 
and $\nsupp(\v+\u-\u')=J\in \NS_\w(\v)^c$.
By Lemma \ref{lem:a_u:multi},
the part of the lowest total degree 
in 
$$
\prod_{i\in I_\0\setminus K}(s_1b^{(1)}_i+s_2b^{(2)}_i+\cdots+s_lb^{(l)}_i)
\partial^{\u'_+}a_\u(\s)x^{\v+\s\b+\u}
$$
is a nonzero constant multiple of
\begin{eqnarray}
\label{eqn:LowestPart}
&& \prod_{i\in (I_\0\setminus K)\setminus (I_\0\setminus I)}
(s_1b^{(1)}_i+s_2b^{(2)}_i+\cdots+s_lb^{(l)}_i)\\
&&\qquad\times
\prod_{i\in J\setminus I}(s_1b^{(1)}_i+s_2b^{(2)}_i+\cdots+s_lb^{(l)}_i)
\nonumber\\
&&\qquad\times
\prod_{i\in I\setminus I_\0}(s_1b^{(1)}_i+s_2b^{(2)}_i+\cdots+s_lb^{(l)}_i).
\nonumber
\end{eqnarray}
We have
$(I_\0\setminus K)\setminus (I_\0\setminus I)=I_\0\cap I\setminus K$.
Since the three sets
$I_\0\cap I\setminus K$, $J\setminus I$, and $I\setminus I_\0$
are disjoint and their union equals
$I\cup J \setminus K$,
we see that \eqref{eqn:LowestPart} equals
\begin{eqnarray*}
&&
\prod_{i\in I\cup J\setminus K}(s_1b^{(1)}_i+s_2b^{(2)}_i+\cdots+s_lb^{(l)}_i)\\
&=&
\sum_{I\cup J\setminus K=\coprod_{j=1}^l L_j}
\prod_{j=1}^l\prod_{i\in L_j}s_jb^{(j)}_i\\
&=&
\sum_{\sum_{j=1}^l p_j=|I\cup J\setminus K|}
\sum_{|L_1|=p_1,\ldots,|L_l|=p_l}
\prod_{j=1}^lb^{(j)}_{L_j}
s_1^{p_1}s_2^{p_2}\cdots s_l^{p_l}.
\end{eqnarray*}

Hence, as in the proof of Theorem \ref{thm:1},
we have $(1)$.

Furthermore, if the assumption of $(2)$ is satisfied,
then
the $s_1^{p_1}s_2^{p_2}\cdots s_l^{p_l}$ part of each
$\partial^{\u'_+}\widetilde{F}(x,\s)$ is zero.
Hence we have $(2)$.
\end{proof}

\begin{remark}
\label{rmk:Method2}
The proof of Theorem \ref{thm2} is again valid
for any $N\subseteq \NS_\w(\v)$ with $N\ni\nsupp(\v)$
and $N^c:=(\NS_\w(\v)\cup \NS_\w(\v)^c)\setminus N$
in place of $\NS_\w(\v)$ and $\NS_\w(\v)^c$.
\end{remark}

\begin{example}[Continuation of Example \ref{running3-2}]
\label{running3-3}

Let 
$A=
\begin{pmatrix}
1 & 1 & 1 & 1 & 1\\
-1 & 1 & 1 & -1 & 0\\
-1 & -1 & 1 & 1 & 0
\end{pmatrix}
$, and take $\w$ as before.

Let
$\displaystyle
\bbeta=
\begin{pmatrix}
1\\
0\\
0
\end{pmatrix}.
$
Then $\v:=(0, 0, 0, 0,1)^T$ is the unique fake exponent.

$$
\NS_\w(\v)=\{ \emptyset, \{ 5\}\},
$$
$$
\NS_\w(\v)^c=\{ \{ 1,3\}, \{ 2,4\}, \{ 1,3,5\}, 
\{ 2,4,5\}, \{ 1,2,3,4\}\},
$$
Hence
$M=2$, $m=0$, $I_\0=\emptyset$,
and
$$
(\partial^j_s F_\b(x,s))_{|s=0} \quad (j=0,1)
$$
are solutions for any $\b\in L$ by Theorem \ref{thm:1}.

Let $\b^{(1)}, \b^{(2)}$ be the column vectors of $B$
in Example \ref{running3-2}.
We have $K=\emptyset$.
The $(I\cup J\setminus K)$'s with $|I\cup J|=M$ are 
$$
\{ 1,3\}, \{ 2,4\}.
$$
Let $p_1=p_2=1$.
Then
\begin{eqnarray*}
&&
b^{(1)}_{1}b^{(2)}_{3}+b^{(1)}_{3}b^{(2)}_{1}=1\cdot 0+1\cdot 0=0\quad
(I\cup J\setminus K=\{ 1,3\}),\\
&&
b^{(1)}_{2}b^{(2)}_{4}+b^{(1)}_{4}b^{(2)}_{2}=0\cdot 1+0\cdot 1=0\quad
(I\cup J\setminus K=\{ 2,4\}).
\end{eqnarray*}
Hence, by Theorem \ref{thm2} (2),
$(\partial_{s_1}\partial_{s_2}{F}_{\b_1,\b_2}(x,\s))_{|\s=\0}$
is also a solution.

Note that $\nsupp(\v+\u)=\emptyset \Leftrightarrow \u=\0$, and
$a_\0(\s)=1$.
We have
$$
{F}_{\b_1,\b_2}(x,\0)=x_5,
$$
\begin{eqnarray*}
(\partial_{s_1}{F}_{\b_1,\b_2}(x,\s))_{|\s=\0}
&=&
(\partial_{s}{F}_{\b_1}(x,s))_{|s=0}\\
&=&
x_5(\log x^{\b_1})+\sum_{\nsupp(\v+\u)=\{ 5\}}(\partial_{s_1}a_\u)(0)x^{\v+\u},
\end{eqnarray*}
\begin{eqnarray*}
(\partial_{s_2}{F}_{\b_1,\b_2}(x,\s))_{|\s=\0}
&=&
(\partial_{s}{F}_{\b_2}(x,s))_{|s=0}\\
&=&
x_5(\log x^{\b_2})+\sum_{\nsupp(\v+\u)=\{ 5\}}(\partial_{s_2}a_\u)(0)x^{\v+\u},
\end{eqnarray*}
\begin{eqnarray*}
(\partial_{s_1}\partial_{s_2}{F}_{\b_1,\b_2}(x,\s))_{|\s=\0}
&=&
x_5(\log x^{\b_1})(\log x^{\b_2})\\
&&+
\sum_{\nsupp(\v+\u)=\{ 5\}}
(\partial_{s_1}a_\u)(0)(\log x^{\b_2})x^{\v+\u}\\
&&+
\sum_{\nsupp(\v+\u)=\{ 5\}}
(\partial_{s_2}a_\u)(0)(\log x^{\b_1})x^{\v+\u}\\
&&+
\sum_{\nsupp(\v+\u)=\{ 5\}}
(\partial_{s_1}\partial_{s_2}a_\u)(0)x^{\v+\u}.
\end{eqnarray*}
\end{example}

\bigskip

Department of Mathematics, Faculty of Science

Hokkaido University

Sapporo, 060-0810, Japan

e-mail: saito@math.sci.hokudai.ac.jp

\end{document}